\documentclass[11pt]{article}
\usepackage{mathrsfs}
\usepackage{mathrsfs}
\usepackage{mathrsfs}
\usepackage{amsfonts}
\usepackage{amssymb}
\usepackage{amsmath}
\usepackage{amsthm}
\usepackage{color}
\theoremstyle{plain}
\newtheorem{thm}{Theorem}[section]
\newtheorem{cor}[thm]{Corollary}
\newtheorem{lem}[thm]{Lemma}
\newtheorem{prop}[thm]{Proposition}

\newtheorem{rem}[thm]{Remark}
\theoremstyle{definition}

\def\no{\noindent}
\def\bs{\bigskip}

\def\ran{\operatorname{ran}}
\begin{document}
\title{{\bf {Fredholm index of Toeplitz  pairs  with $H^{\infty}$ symbols }}
}
\author{{\normalsize Penghui Wang  and  Zeyou Zhu } \\
{ School of Mathematics, Shandong University,} {\normalsize Jinan, 250100, China} }

\date{}
\maketitle

\begin{abstract}
In the present paper, we characterize the Fredholmness of  Toeplitz pairs  on Hardy space over the bidisk with the bounded holomorphic symbols, and hence we obtain the index formula for such Toeplitz pairs.  The key to obtain the Fredholmness of such Toeplitz pairs is the $L^p$
solution of Corona Problem over $\mathbb{D}^2$.
\end{abstract}

\bs


\bs

\no{\bf Key Words:} Toeplitz pair, Fredholm index, Hardy space over the polydisc.

\numberwithin{equation}{section}
\newtheorem{theorem}{Theorem}[section]
\newtheorem{lemma}[theorem]{Lemma}
\newtheorem{proposition}[theorem]{Proposition}
\newtheorem{corollary}[theorem]{Corollary}

\section{Introduction}
\indent\indent

 Let $\mathbb D^2\subseteq \mathbb C^2$ be the unit polydisc, $H^2(\mathbb D^2)$ be the Hardy space over $\mathbb{D}^2,$ and $H^\infty(\mathbb D^2)$ be the space of bounded holomorphic functions. We will study the Fredholmness and Fredholm index of Toeplitz pair $(T_{f_1},T_{f_2})$ on $H^2(\mathbb D^2)$ with $f_i\in H^\infty(\mathbb D^2)$.

The Fredholmness of Toeplitz pair $(T_{f_1},T_{f_2})$ with bounded holomorphic symbols on the Bergman space $L^2_a(\mathbb B_n)$ over  the unit ball $\mathbb B_n$ was studied by Putinar\cite{PM}, which was
generalized to the case of strongly pseudoconvex domains with $C^3$ boundary by Andersson and Sandberg\cite{AS}.
For Toeplitz tuple $(T_{\varphi_1},\cdots,T_{\varphi_m})$ on the Bergman space $L^2_a(\Omega)$ on a strongly pseudoconvex domains $\Omega\subseteq \mathbb C^n$ with symbols $\varphi_i\in C(\overline{\Omega})$, the Fredholm index was studied by Guo\cite{GK}, in which the essential commutativity shown in \cite{SSU} of the Toepliz algebra played the key role.

 In \cite{GW}, Guo and the first author studied the Fredholm index of Toeplitz pairs on $H^2(\mathbb D^2)$ with rational inner symbols, in which the Fredholmness can be deduced by Spectral Mapping Theorem directly. However, for Toeplitz
pair $(T_{f_{1}},T_{f_{2}})$ with bounded holomorphic symbols, the Spectral Mapping Theorem does not work.
In the present note, the Fredholmness will be obtained by using the idea in the proof for existence of $L^p$
solution of Corona Problem over $\mathbb{D}^2$  \cite{LK1,LK2}. In \cite{DS}, Douglas and Sarkar highlighted the relationship between the Corona
Theorem and Fredholmness of Toeplitz tuples over the ball and the disk, and
this connection was also emphasized by \cite{AS}.
To state the result, for $0<r<1$, set $$\mathbb{D}_r=\{z,|z|<r\},\quad\mathbb{D}_r^n=\mathbb{D}_r\times\cdots \times\mathbb{D}_r,\quad \text{and}\quad \mathbb{U}_r^n=\mathbb{D}^n\setminus \mathbb{D}_r^n,$$ and we have the following result.

\begin{thm}\label{th1}
For $f_{1}, f_{2}\in H^\infty(\mathbb D^2)$, then the Toeplitz pair $\left(T_{f_{1}}, T_{f_{2}}\right)$ on $H^{2}\left(\mathbb{D}^{2}\right)$ is Fredholm if and only if there is $0<r<1$ and  $c>0$ such that $\left|f_{1}(z)\right|^{2}+\left|f_{2}(z)\right|^{2} \geqslant c$ on ${\mathbb U}_r^2$, and in this case, the Fredholm index of the Toeplitz pair is given by $$\operatorname{Ind} \left(T_{f_{1}}, T_{f_{2}}\right)=-\operatorname{dim} H^{2}(\mathbb{D}^{2}) /( f_{1} H^{2}\left(\mathbb{D}^{2}\right)+f_{2} H^{2}(\mathbb{D}^{2})).
$$
\end{thm}

Such a result can be seen as a polydisc analogue to the result of Putinar\cite{PM}. The difficulty to prove the above result is to characterize when $f_1 H^2(\mathbb D^2)+f_2 H^2(\mathbb D^2)$ is of finite codimension, which was solved by using the characteristic space theory of polynomials in \cite{GW} for the case that $f_1$ and  $f_2$ are rational inner functions.

   For Toeplitz tuple $(T_{p_1},\cdots,T_{p_n})$ with symbols in ploynomial ring $\mathbb{C}[z_1,\cdots,z_n],$
 the Fredholm index was given in Kaad and Nest\cite[Corollary 9.4]{KJ}.
If $\partial\mathbb D^n\cap(\bigcap\limits_{i=1}^n Z(p_i))=\emptyset$, then $(T_{p_1},\cdots,T_{p_n})$
is Fredholm, and
\begin{eqnarray}\label{thm:poly-index}
\operatorname{Ind}(T_{p_{1}},\cdots,T_{p_{n}})=-\operatorname{deg}
(0,P_n,\mathbb{D}^n),
\end{eqnarray}
where $P_n=(p_1,\cdots,p_n):\mathbb D^n\rightarrow\mathbb C^n$ is the polynomial map.

\begin{rem}\label{rem1}
 Let $f_i\in A(\mathbb D^2)$. Suppose that $Z(f_1)\cap Z(f_2)\cap\partial\mathbb D^2=\emptyset$, then $(T_{f_1}, T_{f_2})$ is Fredholm. Moreover, for $i=1,2$ we can find a sequence of polynomials $p_{1,n}$ and $p_{2,n}$, such that $p_{i,n}$ converge to $f_i$ uniformly. Hence, for $n$ large enough, $$Z(p_{1,n})\cap Z(p_{2,n})\cap\partial\mathbb D^2=\emptyset.$$ Set the map $P_n=(p_{1,n},p_{2,n})$ and $F=(f_1, f_2)$.  Since both the Fredholm index and $deg(0,F,\mathbb D^2)$ are topological invariants, combining with (\ref{thm:poly-index}), we have
$$
{\rm Ind}(T_{f_1}, T_{f_2})=-\deg(0, F, \mathbb D^2).
$$
\end{rem}

The paper is arranged as follows. In section 2, we will characterize the Fredholmness of Toeplitz pairs with bounded holomorphic symbols. In section 3, the case that the symbols are in $L^\infty$ will be considered.

\section{The Fredholm index of Toeplitz pairs with bounded holomorphic symbols}
The present section is devoted to the Fredholmness and the Fredholm index of Toeplitz pairs on
$H^2(\mathbb D^2)$ with bounded holomorphic symbols.
First, we recall the notation of the Koszul complex for a commuting pair. Let $\left(T_{1}, T_{2}\right)$ be a commuting  pair of operators on $H$; the Koszul complex associated with $\left(T_{1}, T_{2}\right)$ was introduced in \cite{TJ},
$$
0 \rightarrow H \stackrel{d_{1}}{\rightarrow} H \oplus H \stackrel{d_{2}}{\rightarrow} H \rightarrow 0,
$$
where the boundary operators $d_{1}, d_{2}$ are given by
$$
d_{1}(\xi)=\left(-T_{2} \xi, T_{1} \xi\right), d_{2}\left(\xi_{1}, \xi_{2}\right)=T_{1} \xi_{1}+T_{2} \xi_{2}, \quad \text { for } \xi, \xi_{1}, \xi_{2} \in H .
$$
Obviously, $d_{2} d_{1}=0$. The commuting pair $\left(T_{1}, T_{2}\right)$ is called Fredholm \cite{CR} if
\begin{eqnarray}
\label{Fredholm}\mathcal{H}_{0}=\operatorname{ker}\left(d_{1}\right), \mathcal{H}_{1}=\left.\operatorname{ker}\left(d_{2}\right)\right/\operatorname{Ran}\left(d_{1}\right) \text{ and } \mathcal{H}_{2}=\left.H\right/\operatorname{Ran}\left(d_{2}\right)
\end{eqnarray}
are all of finite dimension, and in this case, the Fredholm index of $T$ is defined by $$\operatorname{Ind}T=-\operatorname{dim} \mathcal{H}_{0}+\operatorname{dim} \mathcal{H}_{1}-\operatorname{dim} \mathcal{H}_{2}.$$

The following easy lemma may be well-known before, and the proof is left as an exercise for the readers.
\begin{lem}\label{lem2.1}
Let $T_{1}, T_{2},\cdots,T_{n}$ be bounded operators on a Hilbert space $H$, then $T_{1}H+\cdots+T_{n}H$ is closed if and only if the operator
$T_{1}T_{1}^{*}+\cdots+T_{n}T_{n}^{*}$ has the closed range, and in this case
$$\operatorname{Ran}\left(T_{1}T_{1}^{*}+\cdots+T_{n}T_{n}^{*}\right)=T_{1}H+\cdots+T_{n}H.$$
\end{lem}

To continue, we need the following lemma, the proof is inspired by the proof of the $L^p$-solution of the corona problem  \cite{LK1,LK2}.

\begin{lem}\label{le1} Suppose $f_{1}, f_{2} \in H^{\infty}\left(\mathbb{D}^{2}\right)$. Then
$$\operatorname{dim} H^{2}\left(\mathbb{D}^{2}\right) \left/ (f_{1}H^{2}\left(\mathbb{D}^{2}\right)+f_2 H^2(\mathbb D^2))<\infty\right.$$
 if and only if
there is $0<r<1$ and  $\delta>0$ such that $|f_1|^2+|f_2|^2>\delta$ on $\mathbb{U}_r^2.$
\end{lem}

 \begin{proof}

 {\it Necessity.}  Since $f_{1}H^{2}\left(\mathbb{D}^{2}\right)+f_2 H^2(\mathbb D^2)$ is of finite codimension, by Lemma \ref{lem2.1}, we have
  $$
Ran\left(T_{f_1}T_{f_1}^*+T_{f_2}T_{f_2}^*\right)=f_{1}H^{2}\left(\mathbb{D}^{2}\right)+f_2 H^2(\mathbb D^2),
  $$
 then $T_{f_1}T_{f_1}^*+T_{f_2}T_{f_2}^*$ is a positive Fredholm operator. Hence there exists a positive invertible operator $X$ and a compact operator $K$ such that $$T_{f_1}T_{f_1}^*+T_{f_2}T_{f_2}^*=X+K.$$
 For $\lambda=(\lambda^{(1)},\lambda^{(2)})\in \mathbb{D}^2,$ let
$k_\lambda=\frac{\sqrt{1-|\lambda^{(1)}|^2}\sqrt{1-|\lambda^{(2)}|^2}}{(1-\overline{\lambda^{(1)}}z_1)(1-\overline{\lambda^{(2)}}z_2)}$ be the normalized reproducing kernel of $H^2(\mathbb D^2)$ which converges to $0$ weakly as
 $\lambda \rightarrow \partial \mathbb{D}^{2}.$ It follows that there is positive constant $c>0$ such that

 $$\lim\limits _{\lambda\rightarrow \partial \mathbb{D}^{2}}\langle K k_{\lambda}, k_{\lambda}\rangle \rightarrow 0 , ~\text{and}~\liminf\limits _{\lambda\rightarrow \partial \mathbb{D}^{2}} \left\langle X k_{\lambda}, k_{\lambda}\right\rangle =c.$$
Therefore for $\lambda_0 \in \partial \mathbb{D}^{2},$
  $$
\liminf_{\lambda\rightarrow\lambda_0}|f_1(\lambda)|^2+|f_2(\lambda)|^2= \liminf_{\lambda\rightarrow\lambda_0} ||T_{f_{1}}^{*} k_{\lambda}||^{2}+||T_{f_{2}}^{*} k_{\lambda}||^{2}=\liminf_{\lambda\rightarrow\lambda_0}(\langle X k_{\lambda}, k_{\lambda}\rangle+\langle K k_{\lambda}, k_{\lambda}\rangle)\geq c.
 $$
 Thus there is $0<r<1$ and $\delta>0$ such that
 $|f_1(\lambda)|^2+|f_2(\lambda)|^2\geq \delta$ for $\lambda \in \mathbb{U}_r^2.$

{\it Sufficiency.} Notice that $ V=Z(f_1)\cap Z(f_2)$ is a compact zero subvariety of $\mathbb D^2$, then by \cite[Theorem 14.3.1]{Ru} it is a finite set.
Let $\mathcal{O}$ be the sheaf of germs of analytic functions in $\mathbb{D}^2,$ and $(f_1,f_2)\mathcal{O}$ be the ideal generated by $\{f_1,f_2\}.$ Since the support of the analytic sheaf $\mathcal{O} /\left(f_1, f_2\right) \mathcal{O}$ coincides with the set $V,$ hence the space $\mathcal{O}(\mathbb D^2) /\left(f_1, f_2\right) \mathcal{O}(\mathbb D^2)$ is finite-dimensional, where $\mathcal{O}(\mathbb D^2)$ is the space of holomorphic
 functions over $\mathbb D^2.$ Let
  $$
  M=\{f \in H^{2}\left(\mathbb{D}^{2}\right): f=f_1 h_1+f_2h_2 ~\text{for some} ~h_1, h_2 \in \mathcal{O}(\mathbb D^2) \}.
  $$
To show that the subspace $f_1 H^2(\mathbb D^2)+f_2H^2(\mathbb D^2)$ is closed and of finite codimension, it suffices to show that
\begin{eqnarray}\label{M}
M\subset f_1 H^2(\mathbb D^2)+f_2 H^2(\mathbb D^2).
\end{eqnarray}
In fact, \eqref{M} implies that $$A:  H^2(\mathbb D^2)\left/ \sum \limits_{i=1}^2 f_iH^2(\mathbb D^2)\right. \longrightarrow \mathcal{O}(\mathbb D^2) \left/ \left(f_1, f_2\right) \mathcal{O}(\mathbb D^2)\right.$$ is injective.
Fix $g \in M$, then $g=f_1 h_1+f_2h_2$ for some $h_1, h_2 \in \mathcal{O}(\mathbb D^2).$
Let $\phi:[0,1]\times [0,1] \rightarrow[0,1]$ be a smooth function with
$$\phi(t_1, t_2)=1 ~\text{for}~ (t_1,t_2)\in [0,r+\frac{1-r} {3}] \times [0,r+\frac{1-r} {3}],$$ and $$\phi(t_1, t_2)=0 ~\text{for}~ t_1 \geq r+\frac{2(1-r)}{3}~ \text{or}~ t_2 \geq r+\frac{2(1-r)}{3}.$$ Set $$\chi(z_1,z_2)=\phi(|z_1|,|z_2|),$$
and
  \begin{eqnarray}\label{d-bar1}
\varphi_{j}=
\begin{cases}
\frac{g \bar{f_{j}}}{|f_{1}|^{2}+|f_{2}|^{2}},& |f_{1}|^{2}+|f_{2}|^{2} \neq 0,  \\
0  ,& |f_{1}|^{2}+|f_{2}|^{2} = 0 .
  \end{cases}
\end{eqnarray}
It is easy to see that the functions $l_j=\chi h_j+(1-\chi) \varphi_j$ are smooth and satisfy the identity
\begin{eqnarray}\label{smooth}
f_1 l_1+f_2 l_2=g.
\end{eqnarray}
 In general, $l_{1}$ and $l_{2}$ are not holomorphic. As in \cite{LK1}, we will use the technique on normal family of holomorphic functions.
Set
 $$G_i=\frac{1}{g}\left[l_{1} \frac{\partial l_{2}}{\partial \bar{z}_{i}}-l_{2} \frac{\partial l_{1}}{\partial \bar{z}_{i}}\right],$$
i=1,2.
 By straightforward calculations, $G_i\in C^\infty({\mathbb{D}^2}),i=1,2,$
 $$
 \frac{\partial G_{1}}{\partial \bar{z}_{2}}=\frac{\partial G_{2}}{\partial \bar{z}_{1}}
$$
in $\mathbb{D}^2,$
and
 $$
 G_1=\frac{1}{g}\left[\varphi_{1} \frac{\partial \varphi_{2}}{\partial \bar{z}_{1}}-\varphi_{2} \frac{\partial \varphi_{1}}{\partial \bar{z}_{1}}\right]=\frac{g\left(\bar{f_{1}} \frac{\partial \bar{f_{2}}}{\partial \bar{z}_{1}}-\bar{f_{2}} \frac{\partial \bar{f_{1}}}{\partial \bar{z}_{1}}\right)}{\left(| f_{1}|^{2}+|f_{2}|^{2}\right)^{2}},
  $$
  $$
  G_2=\frac{1}{g}\left[\varphi_{1} \frac{\partial \varphi_{2}}{\partial \bar{z}_{2}}-\varphi_{2} \frac{\partial \varphi_{1}}{\partial \bar{z}_{2}}\right]=\frac{g\left(\bar{f_{1}} \frac{\partial \bar{f_{2}}}{\partial \bar{z}_{2}}-\bar{f_{2}} \frac{\partial \bar{f_{1}}}{\partial \bar{z}_{2}}\right)}{\left(| f_{1}|^{2}+|f_{2}|^{2}\right)^{2}}$$
   in $\mathbb{U}_{r+\frac{2(1-r)}{3}}^2.$
Notice that
  $$
 \frac{\partial G_{1}}{\partial \bar{z}_{2}}=\frac{\partial G_{2}}{\partial \bar{z}_{1}}
$$
implies that
 the system
\begin{eqnarray}\label{d-bar1}
\left\{
\begin{aligned}
\frac{\partial b}{\partial \bar{z}_{1}}&=G_1,  \\
\frac{\partial b}{\partial \bar{z}_{2}}&=G_2
  \end{aligned}
\right.
\end{eqnarray}
is $\bar{\partial}-$ closed.
 Then
 by the proof in \cite{LK1}, the equations (\ref{d-bar1}) admit a solution $b\in L^2(\mathbb T^2).$ Put \begin{eqnarray}\label{revision-gi} g_{1}=l_{1}+b f_{2}, \quad g_{2}=l_{2}-b f_{1},\end{eqnarray} then $g_i\in H^2(\mathbb D^2)$ and
$$f_{1}g_{1}+f_{2}g_{2}=g.$$
It follows that $M\subset f_1H^2(\mathbb D^2)+f_2H^2(\mathbb D^2)$.

\end{proof}

For $\{f_1,\cdots,f_m\}\subset H^\infty(\mathbb D^n)$, using the same reasoning as the proof of Lemma \ref{le1}, we have that $$\operatorname{dim} H^{2}\left(\mathbb{D}^{n}\right) \left/ \sum\limits_{i=1}^m f_i H^2(\mathbb D^n)<\infty\right.$$
if and only if there is a constant $\delta>0$ and $0<r<1$ such that $$\sum\limits_{i=1}^m |f_i(z)|^2\geq \delta \quad\text{for}\quad z\in \mathbb{U}_r^n.$$

For example, for $m=2, n=3,$ we only prove the sufficiency. Similarly,
$ V=Z(f_1)\cap Z(f_2)$ is compact, and hence is a finite set.
 Then the space $\mathcal{O}(\mathbb D^3) /\left(f_1, f_2\right) \mathcal{O}(\mathbb D^3)$ is finite-dimensional.
 Let
  $$
  M=\{f \in H^{2}\left(\mathbb{D}^{3}\right): f=f_1 h_1+f_2h_2 ~\text{for some} ~h_1, h_2 \in \mathcal{O}(\mathbb D^3) \}.
  $$
Fix $g \in M$, then $g=f_1 h_1+f_2h_2$ for some $h_1, h_2 \in \mathcal{O}(\mathbb D^3).$
Let $\phi:[0,1]^3 \rightarrow[0,1]$ be a smooth function with
$$\phi(t_1, t_2,t_3)=1 ~\text{for}~ (t_1,t_2,t_3)\in [0,r+\frac{1-r} {3}] \times [0,r+\frac{1-r} {3}] \times [0,r+\frac{1-r} {3}],$$ and $$\phi(t_1,t_2,t_3)=0 ~\text{for}~ t_1 \geq r+\frac{2(1-r)}{3}~ \text{or}~ t_2 \geq r+\frac{2(1-r)}{3}~\text{or}~ t_3 \geq r+\frac{2(1-r)}{3}.$$
Set $$\chi(z_1,z_2,z_3)=\phi(|z_1|,|z_2|,|z_3|),$$ 
and
  \begin{eqnarray}
\varphi_{j}=
\begin{cases}
\frac{g \bar{f_{j}}}{|f_{1}|^{2}+|f_{2}|^{2}},& |f_{1}|^{2}+|f_{2}|^{2} \neq 0,  \\
0  ,& |f_{1}|^{2}+|f_{2}|^{2} = 0 .
  \end{cases}
\end{eqnarray}
It is easy to see that the functions $l_j=\chi h_j+(1-\chi) \varphi_j$ are smooth and satisfy the identity $$f_1 l_1+f_2 l_2=g.$$
Set
 $$G_i=\frac{1}{g}\left[l_{1} \frac{\partial l_{2}}{\partial \bar{z}_{i}}-l_{2} \frac{\partial l_{1}}{\partial \bar{z}_{i}}\right],$$
i=1,2,3.
Considering the equation
\begin{eqnarray}\label{d-bar}
\left\{
\begin{aligned}
\frac{\partial b}{\partial \bar{z}_{1}}&=G_1,  \\
\frac{\partial b}{\partial \bar{z}_{2}}&=G_2,\\
\frac{\partial b}{\partial \bar{z}_{3}}&=G_3,
  \end{aligned}
\right.
\end{eqnarray}
by \cite[Section 3]{LK2}, there exists a solution $b\in L^2(\mathbb T^3).$ Put \begin{eqnarray}\label{revision-gi} g_{1}=l_{1}+b f_{2}, \quad g_{2}=l_{2}-b f_{1},\end{eqnarray} then $g_i\in H^2(\mathbb D^3)$ and
$$f_{1}g_{1}+f_{2}g_{2}=g.$$
It follows that $M\subset f_1H^2(\mathbb D^3)+f_2H^2(\mathbb D^3)$.

Now, we can prove the main result Theorem \ref{th1}.

\noindent$The ~proof ~of~ Theorem~ \ref{th1}.$
  The necessity is an easy application of Lemma \ref{le1}. In fact,  suppose that
$\left(T_{f_{1}}, T_{f_{2}}\right)$ is Fredholm. By the definition, the subspace $f_1 H^2(\mathbb D^2)+f_2 H^2(\mathbb D^2)$ is closed and of finite codimension. By Lemma \ref{le1},  there is $0<r<1$ and  $c>0$ such that $$\left|f_{1}(z)\right|^{2}+\left|f_{2}(z)\right|^{2} \geqslant c,\,\text{ for } z\in \mathbb{U}_r^2.$$

 For the other direction, let $\mathcal{H}_{0}, \mathcal{H}_{1}, \mathcal{H}_{2}$ be defined as in \eqref{Fredholm} for $\left(T_{f_{1}}, T_{f_{2}}\right)$ and suppose there is $0<r<1$ and  $c>0$ such that $\left|f_{1}(z)\right|^{2}+\left|f_{2}(z)\right|^{2} \geqslant c$ on $\mathbb{U}_r^2$, then by Lemma \ref{le1}, the subspace $f_{1} H^{2}\left( \mathbb{D}^{2}\right)+f_{2} H^{2}\left( \mathbb{D}^{2}\right)$ is closed and $$\dim{\cal H}_2=\operatorname{dim} H^{2}\left( \mathbb{D}^{2}\right) / (f_{1} H^{2}\left( \mathbb{D}^{2}\right)+f_{2} H^{2}\left( \mathbb{D}^{2}\right))<\infty.$$
Next, for any $\zeta\in H^2(\mathbb D^2)$,
 $$
\begin{aligned}
\|d_1\zeta\|^2&=\left\|-T_{f_{2}} \zeta\right\|^{2}+\left\|T_{f_{1}} \zeta\right\|^{2}\\&=\sup _{0 \leq r <1} \int_{\mathbb{T}^{2}}\left|f_{2}\zeta (r \xi)\right|^{2} d m+\sup _{0 \leq r <1} \int_{\mathbb{T}^{2}}\left|f_{1}\zeta (r \xi)\right|^{2} d m\\
&\geq \sup _{0 \leq r<1} \int_{\mathbb{T}^{2}}(\left|f_{2}(r \xi)\right|^{2}+\left|f_{1}(r \xi)\right|^{2})|\zeta (r \xi)|^{2} d m \\& \geqslant  c \sup _{0 \leq r <1} \int_{\mathbb{T}^{2}}\left|\zeta (r \xi)\right|^{2} d m
=c\|\zeta\|^{2},
\end{aligned}
$$
which implies that the boundary operator $d_{1}$ is  injective and has closed range, and hence ${\cal H}_0$ is trivial.
Moreover, it is easy to see that $(\zeta_1,\zeta_2)\in \mathcal{H}_{1}$ if and only if $(\zeta_1,\zeta_2)$ solves  the following equations
$$
\left\{\begin{aligned}
& T_{f_{1}} \zeta_{1}+T_{f_{2}} \zeta_{2}=0 , \\
&-T_{f_{2}}^{*} \zeta_{1}+T_{f_{1}}^{*} \zeta_{2}=0.
\end{aligned}\right.
$$
Since $T_{f_{1}} \zeta_{1}+T_{f_{2}} \zeta_{2}=0$, hence  $\frac{-\zeta_{2}}{f_{1}}=\frac{\zeta_{1}}{f_{2}}$ on $\mathbb{D}^{2} \setminus\left( Z\left(f_{1}\right) \cup Z\left(f_{2}\right)\right)$. Set $$\phi=\frac{-\zeta_{2}}{f_{1}}=\frac{\zeta_{1}}{f_{2}},$$
which $\phi$ can be holomorphically extended to $\mathbb{D}^{2} \setminus\left( Z\left(f_{1}\right) \bigcap Z\left(f_{2}\right)\right)$ naturally. Notice that $Z\left(f_{1}\right) \cap Z\left(f_{2}\right) \cap \mathbb{D}^{2}$ is a compact holomorphic subvariety of $\mathbb D^2$, and hence $Z\left(f_{1}\right) \cap Z\left(f_{2}\right) \cap \mathbb{D}^{2}$ is a finite set. By Hartogs' Theorem,  $\phi$ can be holomorphically  extended to $\mathbb{D}^{2}$.
  Furthermore, it is easy to see that there exists $0<s<1$ and $\varepsilon>0$ such that for any $1>r>s$,
 $$
 \left|f_{1}(r \xi)\right|^{2}+\left|f_{2}(r \xi)\right|^{2}>\varepsilon, \quad \text{for all } \xi\in\mathbb T^2.
 $$
  Therefore, for all $\xi \in \mathbb{T}^{2},$
  $$|\phi(r \xi)|^{2}=\frac{\left|-\zeta_{2}(r \xi)\right|^{2}}{\left|f_{1}(r \xi)\right|^{2}}=\frac{\left|\zeta_{1}(r \xi)\right|^{2}}{\left|f_{2}(r \xi)\right|^{2}}
=\frac{\left|-\zeta_{2}(r \xi)\right|^{2}+\left|\zeta_{1}(r \xi)\right|^{2}}{\left|f_{1}(r \xi)\right|^{2}+\left|f_{2}(r \xi)\right|^{2}}<\frac{\left|-\zeta_{2}(r \xi)\right|^{2}+\left|\zeta_{1}(r \xi)\right|^{2}}{\varepsilon}.$$ It follows that $\phi \in H^{2}\left(\mathbb{D}^{2}\right)$, and hence
$$
\zeta_{2}=-f_{1} \phi, ~\zeta_{1}=f_{2} \phi.
$$
  Combining with that $-T_{f_{2}}^{*} \zeta_{1}+T_{f_{1}}^{*} \zeta_{2}=0$, we have
   $
   (T_{f_{1}}^{*} T_{f_{1}}+T_{f_{2}}^{*} T_{f_{2}}) \phi=0,
   $
   thus $$\|\zeta_1\|^2+\|\zeta_2\|^2=\|f_1\phi\|^2+\|f_2\phi\|^2=\left\langle(T_{f_{1}}^{*} T_{f_{1}}+T_{f_{2}}^{*} T_{f_{2}}) \phi, \phi\right\rangle=0,$$ and hence $(\zeta_1,\zeta_2)=0$. This implies that $\mathcal{H}_{1}=0$.
    Thus $\left(T_{f_{1}}, T_{f_{2}}\right)$ is Fredholm. It follows that
$$
\operatorname{Ind}\left(T_{f_{1}}, T_{f_{2}}\right)=-\operatorname{dim} \mathcal{H}_{2}=-\operatorname{codim}\left(f_{1} H^{2}\left(\mathbb{D}^{2}\right)+f_{2} H^{2}\left(\mathbb{D}^{2}\right)\right).
$$
$\hfill \square$

As an easy application, we have the following corollary.

\begin{cor}\label{co1}Let $f_{i}\in H^{\infty}\left(\mathbb{D}^{2}\right)$, $i=1,2$, and $F=(f_1,f_2):\mathbb D^2 \to\mathbb C^2$ be the corresponding map, then the essential spectrum
$$\sigma_{e}\left(T_{f_{1}}, T_{f_{2}}\right)=\bigcap \limits_{0<r<1}\overline{F\left(\mathbb{U}_{r}^2\right)}.$$
\end{cor}
\begin{proof}
It follows from Theorem \ref{th1} and the definition of $\sigma_{e}$ that $(\lambda_1,\lambda_2)\not\in\sigma_{e}(T_{f_1},T_{f_2})$   if and only if there is a $\delta>0$ and $0<r<1$ such that
$$
|\lambda_1-f_1(z)|+|\lambda_2-f_2(z)|>\delta,\quad\text{for}\quad z \in \mathbb{U}_{r}^2,
$$
which is equivalent to saying that $(\lambda_1,\lambda_2)\not \in \overline{F(\mathbb{U}_{r}^2)},$ for some $0<r<1$.
\end{proof}

\section{Fredholmness of Toeplitz tuples with symbols in $L^\infty$}

In the present section, we will consider Toeplitz tuple $(T_{f_1},T_{f_2},\cdots,T_{f_n})$ with symbols in $L^\infty.$
It is easy to see that in general $(T_{f_1},T_{f_2},\cdots,T_{f_n})$ is not essentially commuting.
To get the commutativity of the Toeplitz tuple, let $L_{z_i}^\infty(\mathbb T)$ be the space of
essentially bounded functions over
$\mathbb{T}$ on the variable $z_i,$ and the symbols $f_i\in L_{z_i}^\infty(\mathbb T).$
To avoid the confusion, denoted by $T_{f_i}^{(i)}(resp. ~T_{f_i})$ the Toeplitz
operator on $H^2_{z_i}(\mathbb T)(resp. ~H^2(\mathbb D^n)).$ We have the following Proposition.

\begin{prop}\label{cor2.9}For $f_i\in L_{z_i}^\infty(\mathbb T),$ assume that all $T_{f_i}^{(i)}$ are not invertible. Then the Toeplitz tuple $(T_{f_1},T_{f_2},\cdots,T_{f_n})$ on $H^2(\mathbb D^n)$ is Fredholm if and only if all $T_{f_i^{(i)}}$ are Fredholm, and in this case,
$$
\operatorname{Ind}(T_{f_1},T_{f_2},\cdots,T_{f_n}) =(-)^{n+1}\prod\limits_{i=1}^n\operatorname{Ind} T^{(i)}_{f_i}.
$$
\end{prop}

The proof of Proposition \ref{cor2.9} comes from the following lemma easily, which is a generalization of \cite[Proposition 15.4]{CR}.
 The sufficient part was proved by \cite[Lemma 7.3]{KJ}, by using cohomology, and we will give an elementary proof.

\begin{lem}{\label{lem 2.6}}
Let $\{H_i\}_{i=1}^n$ be a set of Hilbert spaces of infinite dimension and set $$H=H_1\otimes H_2\otimes\cdots\otimes H_n.$$ Suppose all $T_i\in L(H_i)$ are not invertible. Set
$${\widetilde{T_i}}=I_{H_1}\otimes \cdots \otimes I_{H_{i-1}}\otimes T_{i}\otimes I_{H_{i+1}}\otimes\cdots\otimes I_{H_{n}}, \qquad i\leq n,$$ then the tuple $\mathcal{T}=({\widetilde{T}_1},\cdots,{\widetilde{T}_n})$ is a Fredholm tuple on $H$ if and only if all
$T_i$ are Fredholm, and in this case, $$\operatorname{Ind}\mathcal{T}=(-1)^{n+1}\prod\limits_{i=1}^n\operatorname{Ind}{T}_{i}.$$

\end{lem}

\begin{proof}At first, suppose that $\mathcal{T}$ is Fredholm. Notice that all $T_j$ are not invertible, then
$$
H_j/\ran T_j\not=0,\quad\text{or}\quad \ker T_j\not=0.
$$
Let $\Lambda_1=\{i, \ker T_i\not=0\}$ and $\Lambda_2=\{j,\ker T_j=0\},$
 and set $S_i=T_{i}^*$ for $i\in \Lambda_1$ and $S_i=T_i$ for $i\in\Lambda_2$. Since $\mathcal{T}$ is doubly commutative,  that is
 $$
 [\widetilde{T}_i,\widetilde{T}_j]=[\widetilde{T}_i,\widetilde{T}_j^*]=0, \quad 1\leq i,j\leq n.
 $$By \cite[Corollary 3.7]{CR}, the tuple $(\widetilde{S}_1,\cdots,\widetilde{S}_n)$ is Fredholm, where $\widetilde{S}_i$ is defined as same as $\widetilde{T}_i$ for every $i$.
 If follows that as a linear space,
 $$
 H\left/\sum\limits_{j=1}^n\widetilde{S}_i H\right.=\left(H_1/\ran S_1\right)\otimes\cdots \otimes \left(H_n/\ran S_n\right),
 $$
 which is of finite dimension. Since all $H_i/\ran S_i$ are not trivial,
  $$
  \dim H_i/\ran S_i<\infty.
  $$
   It follows that all $\ran S_i$ are closed.  Next, we will show that all $\ker S_j$ are of finite dimension. It suffices to show that $\dim\ker S_j<\infty$ for $j\in\Lambda_1$. For $i\in\Lambda_2$, since $S_i$ is not invertible, we have that $\ker S_i^*\not=0$, and hence for any $1\leq i\leq n$, $\ker S_i^*\not=0$. For $i\in\Lambda_1$, by \cite[Proposition 3.7]{CR} again, the tuple $(\widetilde{S}_1^*,\cdots, \widetilde{S}_{i-1}^*, \widetilde{S}_{i},\widetilde{S}_{i+1}^*,\cdots,\widetilde{S}_n^*)$ is Fredholm. It can be verified that
 $$
 \left(\bigcap\limits_{j\not=i} \ker \widetilde{S}^*_{j}\right)\cap \ker \widetilde{S}_{i}=\ker S_1^*\otimes\cdots\otimes\ker S_{i-1}^*\otimes \ker S_{i}\otimes \ker S_{i+1}^*\otimes\cdots\otimes \ker S_n^*
 $$
is of finite dimension, it follows that all $\ker S_{i}$ are of finite dimension. Therefore all  $S_i$ are Fredholm, and hence all $T_i$ are also Fredholm.

Next,  assume that all $T_i$ are Fredholm. We will prove that the tuple $\mathcal{T}$ is Fredholm.
By \cite[Corollary 3.7]{CR}, the tuple  $({\widetilde{T}_1},\cdots,{\widetilde{T}_n})$ is Fredholm if and only if
$\sum\limits_{i=1}^n{ }^f {\widetilde{T_i}}$ is  Fredholm for every functions $f:\{1, \cdots, n\} \rightarrow\{0,1\},$ where
$$
{ }^f {\widetilde{T_i}}= \begin{cases}{\widetilde{T_i}}^* {\widetilde{T_i}}, & f(i)=0, \\ {\widetilde{T_i}} {\widetilde{T_i}}^*, & f(i)=1.\end{cases}
$$
Now, we will show that  $\widetilde{T}_1\widetilde{T}_1^*+\cdots+\widetilde{T}_n\widetilde{T}_n^*$ is Fredholm. In fact,
\begin{eqnarray*}
&&\widetilde{T}_1H+\cdots+\widetilde{T}_n H\\&=&T_1H_1\otimes H_2\otimes\cdots\otimes H_n+H_1\otimes T_2 H_2\otimes H_3\otimes\cdots\otimes H_n+\cdots+ H_1\otimes\cdots\otimes H_{n-1}\otimes T_n H_n.
\end{eqnarray*}
 Since $T_iH$ is closed, $\widetilde{T}_1H+\cdots+\widetilde{T}_n H$ is closed. By Lemma \ref{lem2.1}, $$\ran\left(\sum\limits_{i=1}^n \widetilde{T}_i\widetilde{T}_i^*\right)=\widetilde{T}_1H+\cdots+\widetilde{T}_n H$$ is closed. Moreover,
$$
\begin{aligned}
  \ker\left( \sum\limits_{i=1}^n{\widetilde{T}}_i{\widetilde{T}_i^*}\right)&=\bigcap\limits_{i=1}^n\ker \widetilde{T}_i^*=\ker T_1^* \otimes\cdots \otimes \ker T_n^*
\end{aligned}
$$
is of finite dimension. This implies that $\sum\limits_{i=1}^n{\widetilde{T}}_i{\widetilde{T}_i^*}$ is Fredholm.
The same discussions as above show that for every function $f:\{1, \cdots, n\} \rightarrow\{0,1\}$, $\sum\limits_{i=1}^n{ }^f {\widetilde{T_i}}$ is Fredholm.

 By \cite[Lemma 7.3]{KJ},
  if $\mathcal{T}=({\widetilde{T}_1},\cdots,{\widetilde{T}_n})$ is Fredholm, then
 $$\operatorname{Ind}\mathcal{T}=(-1)^{n+1}\prod\limits_{i=1}^n\operatorname{Ind}{T}_{i}.$$
 \end{proof}

 For the
remainder of the section, we will consider the Toeplitz tuple $(T_{f_1},\cdots,T_{f_n})$ on
$H^2(\mathbb {D}).$
By Andersson and Sandberg \cite{AS}, for $\{f_{i}\}_{i=1}^n\subset H^\infty(\mathbb D)$,
$\left(
T_{f_{1}},T_{f_{2}},\cdots,T_{f_{n}}\right),$ acting on $H^{2}\left(\mathbb{D}\right),$ is Fredholm if and only if there exists $0<s<1$ and $\delta>0$ such that $$\left|f_{1}(z)\right|^{2}+\left|f_{2}(z)\right|^{2}+\cdots+\left|f_{n}(z)\right|^{2}>\delta,\quad \text{for}\ |z|>s.$$
\begin{prop}\label{prop2.2}
If $\left(T_{f_{1}},T_{f_{2}},\cdots,T_{f_{n}}\right)$ is Fredholm, then $$\operatorname{Ind}(T_{f_{1}},\cdots,T_{f_{n}})=0.$$
\end{prop}
Before giving the proof of Proposition \ref{prop2.2}, we should point out that in Yang \cite{Y}, it was proved that
if $[T_{\phi}^{*},T_{\psi}]$ is compact, then $\operatorname{Ind}\left(T_{\phi},T_{\psi}\right)=0.$

\noindent$\emph{The proof of Proposition }~\ref{prop2.2}.$
Since $\left(T_{f_{1}},T_{f_{2}},\cdots,T_{f_{n}}\right)$ is Fredholm, there are $\delta>0$ and $0<s<1$ such that
$$
|f_1(z)|^2+|f_2(z)|^2+\cdots+|f_n(z)|^2\geq \delta, \quad \text{for}\ |z|>s.
$$
It follows that there is a finite Blaschke product $B$ such that $$f_i=B \tilde{f_{i}}\quad{\text{and}}\quad\sum\limits_{i=1}^n|\tilde{f_{i}}|^2\geq \frac{\delta}{|B|^2}\geq \delta \ \text{on}\ \mathbb D \setminus Z(B). $$
Then $$\sum\limits_{i=1}^n|\tilde{f_{i}}|^2\geq \frac{\delta}{|B|^2}\geq \delta \ \text{on}\ \mathbb D.$$
By Corona Theorem, we have $(T_{\tilde{f_{1}}},\cdots,T_{\tilde{f_{n}}})$ is invertible, and hence $$\operatorname{Ind}(T_{\tilde{f_{1}}},\cdots,T_{\tilde{f_{n}}})=0.$$
Moreover,  since $T_B$ is Fredholm, by \cite[Proposition 11.1]{CR}  for any bounded analytic function $\varphi_i$, the tuple of Toeplitz operators $(T_{\varphi_1},\cdots, T_{\varphi_i},T_{B}, T_{\varphi_{i+1}}, T_{\varphi_n})$ is Fredholm and $$\operatorname{Ind}(T_{\varphi_1},\cdots, T_{\varphi_i},T_{B}, T_{\varphi_{i+1}}, T_{\varphi_n})=0.$$
By \cite[Propsition 1]{FX},
$$
\operatorname{Ind}(T_{f_1},\cdots,T_{f_n})=\operatorname{Ind}(T_B,T_{f_2},\cdots,T_{f_n})
+\operatorname{Ind}(T_{\tilde{f_1}},T_{f_2},\cdots,T_{f_n})=
\operatorname{Ind}(T_{\tilde{f_1}},T_{f_2},\cdots,T_{f_n}).
$$
By induction, $$\operatorname{Ind}(T_{{f_1}},T_{f_2},\cdots,T_{f_n})=\operatorname{Ind}
(T_{\tilde{f_1}},T_{\tilde{f}_2},\cdots,T_{\tilde{f}_n})=0.$$ $\hfill\qedsymbol$

\noindent\textbf{Acknowledgement.} This work is supported by NSFC (No. 12271298, 11871308).  We would like to thank K. Guo (Fudan University) for valuable discussions on this topic. The authors thank the referee for helpful suggestions which make this paper more readable.

  \noindent{Penghui Wang, School of Mathematics, Shandong University, Jinan 250100, Shandong, P. R. China, Email: phwang@sdu.edu.cn}

  \noindent{Zeyou Zhu, School of Mathematics, Shandong University, Jinan 250100, Shandong, P. R. China, Email: 201911795@mail.sdu.edu.cn}


\begin{thebibliography}{99}
\bibitem{AS}M. Andersson and S. Sandberg, \emph{The essential spectrum of holomorphic Toeplitz operators on $H^p$ spaces,} Studia Math. 154 (2003), no. 3, 223-231.


\bibitem{CR}R. Curto, \emph{Fredholm and invertible n-tuples of operators. The deformation problem,}
Trans. Amer. Math. Soc. 266 (1981), no. 1, 129-159.



\bibitem{DS}R.G. Douglas and J. Sarkar, \emph{A note on semi-Fredholm Hilbert modules,} Oper. Theory Adv. Appl. 202, pages 143-150.


\bibitem{FX}X. Fang, \emph{The Fredholm index of a pair of commuting operators,} Geom. Funct. Anal. 16 (2006), no. 2, 367-402.


\bibitem{GK}K. Guo, \emph{Indices of Toeplitz tuples on pseudoregular domains,} Sci. China Ser. A. 43 (2000), no. 12, 1258-1268.
\bibitem{GW}K. Guo and P. Wang, \emph{Defect operators and Fredholmness for Toeplitz pairs with inner symbols,} J. Operator Theory. 58 (2007), no. 2, 251-268.




\bibitem{KJ}J. Kaad and R. Nest, \emph{
A transformation rule for the index of commuting operators,}
J. Noncommut. Geom. 9 (2015), no. 1, 83-119.

\bibitem{LK1}K. Lin, \emph{The $H^p$-corona theorem for the polydisc,} Trans. Amer. Math. Soc. 341 (1994), no. 1, 371-375.


\bibitem{LK2}K. Lin, \emph{$H^p$-solutions for the corona problem on the polydisc in $\mathbb{C}^{n}$,} Bull. Sci. Math. (2) 110 (1986), no. 1, 69-84.



\bibitem{PM}M. Putinar, \emph{On joint spectra of pairs of analytic Toeplitz operators,} Studia Math. 115 (1995), no. 2, 129-134.


  \bibitem{Ru}W. Rudin, \emph{Function theory in the unit ball of $\mathbb{C}^n$,} Springer-Verlag, 1980.


   \bibitem{SSU} N. Salinas, A. Sheu, H. Upmeier, \emph{Toeplitz operators on pseudoconvex domains and foliation C$^*$-algebras,} Ann. of Math. (2) 130 (1989), no. 3, 531-565.
\bibitem{TJ}J. Taylor, \emph{A joint spectrum for several commuting operators,} J. Functional Analysis. 6 (1970), 172-191.






\bibitem{Y}R. Yang, \emph{A trace formula for isometric pairs,} Proc. Amer. Math. Soc. 131 (2003), no. 2, 533-541.


   \end{thebibliography}
\end{document}